\documentclass[11pt]{article}
\usepackage{amssymb,amsmath,latexsym, verbatim, enumerate}
\usepackage{authblk}
\usepackage[hyperfootnotes=false]{hyperref}
\usepackage[mathscr]{eucal}
\usepackage{mathrsfs}
\def\O{\mathcal{O}}

\usepackage{tikz}
\usetikzlibrary{decorations.fractals}
\usepackage{pgfplots}
\usetikzlibrary{lindenmayersystems}

\allowdisplaybreaks

\begin{document}

\begin{doublespace}

\newtheorem{thm}{Theorem}[section]
\newtheorem{lemma}[thm]{Lemma}
\newtheorem{cond}[thm]{Condition} 
\newtheorem{defn}[thm]{Definition}
\newtheorem{prop}[thm]{Proposition}
\newtheorem{corollary}[thm]{Corollary}
\newtheorem{remark}[thm]{Remark}
\newtheorem{example}[thm]{Example}
\newtheorem{conj}[thm]{Conjecture}
\numberwithin{equation}{section}
\def\ee{\varepsilon}
\def\qed{{\hfill $\Box$ \bigskip}}

\def\b{\mathfrak{b}}
\def\eps{\varepsilon}

\def\j{{\boldsymbol j}}

\def\NN{{\cal N}}
\def\AA{{\cal A}}
\def\MM{{\cal M}}
\def\BB{{\cal B}}
\def\CC{{\cal C}}
\def\LL{{\cal L}}
\def\DD{{\cal D}}
\def\FF{{\cal F}}
\def\EE{{\cal E}}
\def\QQ{{\cal Q}}
\def\RR{{\mathbb R}}
\def\R{{\mathbb R}}
\def\L{{\bf L}}
\def\K{{\bf K}}
\def\S{{\bf S}}
\def\A{{\bf A}}
\def\E{{\mathbb E}}
\def\F{{\bf F}}
\def\P{{\mathbb P}}
\def\N{{\mathbb N}}
\def\eps{\varepsilon}
\def\wh{\widehat}
\def\wt{\widetilde}
\def\pf{\noindent{\bf Proof.} }
\def\beq{\begin{equation}}
\def\eeq{\end{equation}}
\def\lam{\lambda}
\def\H{\mathcal{H}}
\def\nn{\nonumber}
\def\L{\mathcal{L}}
\def\eqd{\,{\buildrel d \over =}\,}
\newcommand{\ud}{\mathrm{d}}

\newcommand{\hp}[1]{\textcolor{red}{#1}}

\newcommand\blfootnote[1]{%
  \begingroup
  \renewcommand\thefootnote{}\footnote{#1}%
  \addtocounter{footnote}{-1}%
  \endgroup
}

\title{\Large \bf Small-time heat decay for stable processes on fractal drums}
\author{Hyunchul Park and Yimin Xiao}

\date{\today}
\maketitle


\begin{abstract}
In this paper, we study the spectral heat content for isotropic 
stable processes on fractal drums (namely, open sets with 
fractal boundaries). The spectral heat content for subordinate 
killed Brownian motions by stable subordinators was investigated in 
\cite{PX23}, and the present work serves as a natural extension of 
\cite{PX23} for the spectral heat content for stable processes. Under suitable geometric conditions on the underlying domains, we show that the decay rate of the spectral heat content for stable processes differs substantially from that for subordinate killed Brownian motions when $\alpha=d-\b$, where $\b$ is the interior Minkowski dimension of the boundary of the underlying open set. 
\end{abstract}

\section{Introduction}\label{Introduction}
The spectral heat content (SHC) for a given L\'evy process in a domain $D$ 
measures the total heat remaining in $D$ when the heat particles move according 
to the process with the initial temperature set at one inside $D$ and 
maintained at zero outside $D$. 
The significance of SHC lies in the fact that it encodes both geometric information about the domain and spectral information about the infinitesimal generator of the underlying 
L\'evy  process. In small time, the asymptotic expansion of SHC contains the geometric quantities of the domain such as volume or perimeter in the leading-order terms. On the other hand, the large time behavior of SHC contains the spectral information such as eigenvalues and eigenfunctions of the associated infinitesimal generators of the stochastic process. This places the study of SHC at the intersection of probability theory, geometry, spectral theory, and partial differential equations. 

While the SHC for Brownian motion has been well 
studied since the late 1980s, its analysis for jump-type L\'evy processes began only recently, from the mid-2010s. The small-time 
asymptotics for SHC for subordinate killed Brownian motion by stable 
subordinators was investigated in \cite{PS19} 
and this was extended for more general subordinators in 
\cite{KP23-2}. A closely related process is the subordinate Brownian 
motion, where one observes Brownian motions on a random clock marked 
by a subordinator (a non-decreasing L\'evy process). For a 
non-degenerate subordinator, its sample path is 
discontinuous, and the corresponding subordinate Brownian motion becomes a jump
process.  The SHC for such jump-type L\'evy processes has been studied previously; in 
particular, the small-time asymptotics for SHC for isotropic stable processes were 
investigated in \cite{PS22}, and this was extended to more general L\'evy processes in 
\cite{KP23-3} and \cite{LP25}. 

In \cite{GPS19}, the authors studied the SHC for a one-dimensional symmetric L\'evy 
process in $\R$ with regularly varying characteristic exponent. They showed that when the underlying open set has infinitely many components, the decay rate of SHC is faster than SHC on domains with finitely many components (see \cite[Lemma 4.9]{GPS19}).
A natural question arises concerning the precise decay rate of the SHC on domains with infinitely many components. 
A natural way for a \textit{bounded} open set to possess infinitely many components is to consider the sets of the form given in \eqref{eqn:SS}.

Domains of the form \eqref{eqn:SS} in $\mathbb{R}^d$ typically have fractal boundaries, and such open sets are referred to as fractal drums in \cite{Lapidus17, Lapidus13}. Various problems concerning fractal drums have been investigated, including the Weyl-Berry conjecture for the eigenvalues of the Laplacian on bounded domains in \cite{Lapidus91} and its connection to the Riemann Hypothesis in \cite{LM95}. SHC on fractal drums with respect to Brownian motion was studied in \cite{LV1996}. Subsequently, SHC on fractal drums associated with jump processes was examined. In particular, \cite{PX23} investigated the SHC for subordinate killed Brownian motions by stable subordinators.

The purpose of this paper is to study the SHC on fractal drums for isotropic stable processes. 
A stable process can be realized as 
a subordinate Brownian motion driven by a stable subordinator, where one 
observes the underlying Brownian motion on the random clock determined by a stable subordinator. The resulting process 
evolves through jumps, which is related, but different from the 
process studied in \cite{PX23}. Due to the scaling property of 
stable processes, the SHC for stable processes inherits a 
corresponding scaling property under similitudes.
However, the additivity property of the SHC under disjoint unions, 
crucial in \cite{PX23}, fails for jump processes. This failure 
arises because the time-changed Brownian motion, being observed on 
random clocks, may have wandered back into the domain after the 
underlying Brownian motion has exited it. Consequently, the direct 
application of the renewal theorem becomes nontrivial. 
To address this difficulty, we introduce the auxiliary terms 
$\mathcal{D}(t)$ and $\mathcal{R}(t)$ (see \eqref{eqn:D and R}), 
which quantify the extent to which the SHC for stable processes 
fails to be additive. We then show that the error term $\mathcal{R}(t)$ decays 
exponentially fast, a property that enables the use of the renewal theorem in 
our analysis.

The main results of this paper are Theorems \ref{thm:higher alpha} 
and \ref{thm:lower alpha}. For each theorem, we impose different 
assumptions, Conditions \ref{cond:higher alpha} and  
\ref{cond:lower alpha}, on the underlying open set $G$.
Let $\b$ be the 
interior Minkowski dimension of the boundary $\partial G$ of the domain.
When $\alpha\in (d-\b,2)$, we impose Condition \ref{cond:higher alpha} that 
each rescaled copies of $G$ and the extra component $G_0$ are contained in 
intervals that are disjoint with other parts of $G$ when $d=1$, or they are 
contained in $C^{1,1}$ open sets $\mathcal{O}_{j}$ and $G_0$ is $C^{1,1}$ 
when $d\geq 2$. 
Under these assumptions, we prove in Theorem \ref{thm:higher alpha} that the 
decay rate of SHC is of the order $t^{\frac{d-\b}{\alpha}}$, and 
there are two regimes depending on whether the sequence of logarithms of the 
coefficients of the similitudes is arithmetic or not. A similar phenomenon 
was observed for SHC for the subordinate killed Brownian motions in 
\cite[Theorem 3.2]{PX23}. When $\alpha\in (0,d-\b]$, we impose Condition 
\ref{cond:lower alpha}, which is different from Condition 
\ref{cond:higher alpha}. In this case, we assume that two different 
dimensions, the interior Minkowski dimension and the Hausdorff dimension, of 
$\partial G$ coincide and the Hausdorff measure $\mathcal{H}^{\b}
(\partial G)$ is strictly positive. 
In Theorem \ref{thm:lower alpha} of the present paper, we show that the decay rate of SHC when 
$\alpha\in (0,d-\b]$ is of order $t$, which is a sharp contrast from the results 
in \cite{PX23} when $\alpha=d-\b$. 
In \cite[Theorems 3.4 and 3.5]{PX23}, the decay rates of SHC for subordinate 
killed Brownian motions are of the order $t\ln(1/t)$ or $t$ depending on whether 
$\alpha=d-\b$ or $\alpha\in (0,d-\b)$, respectively.
Furthermore, when $\alpha=d-\b$, the small-time asymptotics of SHC 
for subordinate killed Brownian motions depends on whether the 
sequence of logarithms of the coefficients of the similitudes is 
arithmetic. 
Hence, when $\alpha=d-\b$, the decay rates for SHC for killed stable 
processes and subordinate killed Brownian motions by stable 
subordinators are not comparable.
This is because the boundary of the domain is polar for the stable 
process, and the process never \textit{sees} the boundary. 
Consequently, it cannot detect that the domain has a fractal 
boundary. 
We find this phenomenon very interesting, as the decay rates for SHC for jump processes 
and the corresponding subordinate killed Brownian motions are known to be comparable for smooth domains (see \cite{GPS19, KP23-2, KP23-3, LP25, PS19}).
To the best of the authors' knowledge, this provides the first example in which the decay rates of these two spectral heat contents are not comparable.

The organization of this paper is as follows. 
In Section \ref{Preliminaries}, we provide some preliminaries on 
isotropic stable processes, heat content and  fractal drums,  
and recall the renewal theorem, Theorem 
\ref{thm:RT}. 
In Section \ref{Body}, we prove our main 
results, Theorems \ref{thm:higher alpha} and \ref{thm:lower alpha}. 
While proving these theorems, we discovered minor inaccuracies 
in \cite[Corollary 3.7]{GPS19} when $\alpha\in (d-\b,1)$. See Remark 
\ref{remark:mistake} for details. 
For comparison 
purpose, we also investigate the regular heat content (without 
Dirichlet exterior condition) on fractal drums and show that the 
asymptotics of the regular heat content is not affected by the fractal boundary.
In Section \ref{example}, we provide some concrete examples where
our main theorems can be applied to study the SHC for stable
processes on an open set on $\R$ whose boundary is the ternary Cantor set. We also investigate the regular heat content (without
Dirichlet exterior condition) on the same set and show that the 
regular heat content for this set is the same as the open interval 
$(0,1)$. This is because the boundary, the ternary Cantor set, has 
the Lebesgue measure zero and cannot be observed at any fixed 
time $t$ even though the stable process can observe the boundary in 
a time interval $(0,t)$. 
Finally, we investigate the SHC on the complement of the modified Sierpi\'nski gasket in $\R^2$. 

In this paper, we use $c_i$ ($i = 1, 2, \ldots $) to denote 
constants whose 
values are unimportant and may change from one appearance to another.
The notation $\P_{x}$ stands for the law of the underlying processes 
started at $x\in \R^d$. The notation $D$ is used for a generic open 
set and $G$ will be used for a set given as in \eqref{eqn:SS}. 
We use $\delta_{D}(x)=\text{dist}(x,D^c)$ to denote the distance 
from $x$ into $D^c$. All sets will be assumed to be Borel 
measurable. For any Borel measurable set $A\subset \R^{d}$, 
$m(A)$ will be used for the $d$-dimensional Lebesgue measure 
of $A$ and $\mathcal{H}^{a}(A)$ is for the $a$-dimensional Hausdorff 
measure of $A$. 
For any Borel measurable set $B\subset \R^{d}$ with differentiable boundary $\partial B$, we denote the perimeter of $B$ by $|\partial B|$.

\section{Preliminaries}\label{Preliminaries}
Let $X=\{X_{t}\}_{t\geq 0}$ be an isotropic stable process in $\R^d$ with stability index $\alpha\in (0,2)$. 
The characteristic exponent of $X_t$ is $\E[e^{i\xi X_{t}}]=e^{-t|\xi|^{\alpha}}$ and the L\'evy density is given by $J(x)=\frac{c(d,\alpha)}{|x|^{d+\alpha}}$.
Note that the isotropic stable process $X$ is a subordinate (time-changed) 
Brownian motion. More precisely, for $\alpha \in (0,2)$, let $S^{(\alpha/2)}=\{S_{t}^{(\alpha/2)}\}_{t\geq 0}$ be a stable subordinator 
whose Laplace transform is given by 
\[
\E[e^{-\lam S_{t}^{(\alpha/2)}}]=e^{-t\lam^{\alpha/2}}, \quad \lam>0,
\]
and $W=\{W_{t}\}_{t\geq 0}$ is a Brownian motion in $\R^d$ independent of $S^{(\alpha/2)}$. Then, the process $W_{S^{(\alpha/2)}}=\{W_{S^{(\alpha/2)}_t}\}_{t\geq 0}$ becomes an isotropic stable process in $\R^d$ with index $\alpha$. 

Let $D\subset \R^d$ be an open set. When $\partial D$, the boundary of $D$, is a fractal (by this, we mean the Hausdorff dimension of $\partial D$ has a non-integer value), we will call $D$ a fractal drum as in \cite{Lapidus17, Lapidus13}. 
Let $\tau_{D}=\inf\{t>0: X_{t}\notin D\}$ be the first exit time. 
The spectral heat content $Q_{D}^{(\alpha)}(t)$ with respect to $X$ on the open set $D$ is defined by
\[
Q_{D}^{(\alpha)}(t)=\int_{D}\P_{x}(\tau_{D}>t)dx. 
\]
Note that there is a closely related quantity called the spectral heat content for \textit{subordinate killed Brownian motions} (killing the underlying Brownian motion when it first exits the domain, then doing the time-change).
The spectral heat content $\tilde{Q}_{D}^{(\alpha)}(t)$ for 
subordinate killed Brownian motions with respect to the stable subordinator $S^{(\alpha/2)}$ on $D$ is defined by 
\[
\tilde{Q}_{D}^{(\alpha)}(t)=\int_{D}\P_{x}\Big(\tau_{D}^{(2)}> S_{t}^{(\alpha/2)}\Big)dx,
\]
where $\tau_{D}^{(2)} = \inf\{t>0 : W_{t}\notin D\}$.
The small-time asymptotic behavior of $\tilde{Q}_{D}^{(\alpha)}(t)$ is investigated on $C^{1,1}$ open sets in \cite{PS19} and on fractal drums in \cite{PX23}.

Now we introduce the type of fractal drums that we consider in this paper.
Recall that a map $R:\R^{d}\to \R^{d}$ is called a \textit{similitude} with 
coefficient $r>0$ if $|Rx-Ry|=r|x-y|$ for all $x, y\in \R^{d}$. It is well known 
(cf. e.g., \cite[p.191]{LV1996}) that any similitude is a composition of a homothety
with coefficient $r$, an orthogonal transform, and a translation. 
Let $G_0$ be an open set in $\R^d$ and let $R_j$ ($j = 1, \ldots, N)$ be similitudes 
with coefficients $r_j$.
For each $n \ge 1$, define $\Upsilon_n = \{ {\boldsymbol j} = (j_1, \ldots, 
j_n),\ 1 \le j_i \le N\}$.  
Levitin and Vassiliev \cite{LV1996} defined the open set $G$ by
\beq\label{eqn:SS1}
G= \bigg(\bigcup_{n=1}^{\infty} \bigcup_{ {\boldsymbol j } \in \Upsilon_n}  
{\mathscr R }_{\j}G_0 \bigg)\cup G_{0},
\eeq
where, for every $\j = (j_1, \ldots, j_n) \in \Upsilon_n$, ${\mathscr R }_{\j}$
is the similitude  defined by ${\mathscr R }_{\j} =R_{j_1}\circ \cdots \circ R_{j_n}$. 

It follows from \eqref{eqn:SS1} (see also \cite[Equation (1.8)]{LV1996}) that $G$ satisfies the following property
\beq\label{eqn:SS}
G= \bigg(\bigcup_{j=1}^{N}R_{j}G\bigg) \cup G_{0}.
\eeq
Thus, the family of open sets that satisfy \eqref{eqn:SS} is larger. Moreover, it is more convenient to construct fractal 
drums by choosing appropriate open sets $G_0$ in \eqref{eqn:SS}. 
For example,  the complement of the ternary Cantor set $\mathfrak{D}$ satisfies \eqref{eqn:SS} with $G_0 = (\frac 1 3, \frac 2 3)$,
 $R_1(x) = \frac{x}{3}$ and $R_2(x) = \frac{x}{3} + \frac2 3$; and the complement of the Sierpi\'nski gasket  $\mathfrak{G}$ satisfies
 \eqref{eqn:SS} with $G_0$ being a triangle whose vertices are $(1/4, \sqrt{3}/4)$, $(1/2,0)$ and $(3/4,\sqrt{3}/4)$ and 
 $R_{1}(x)=\frac{x}{2}$, $R_{2}(x)=\frac{x}{2}+(\frac12, 0)$,  and $R_{3}(x)=\frac{x}{2}+(\frac14,\frac{\sqrt{3}}{4})$. See Section 4 for more information. 
Hence, in the present paper, we will use property \eqref{eqn:SS} to define
the open sets $G$. This more general definition allows us to 
assume that all the sets $R_{j}G $, $1 \le j \le N$, and $G_0$ are pairwise disjoint. See Condition 3.1 below. 

As in \cite[Equation (1.7)]{LV1996} we also assume that $\sum_{j=1}^{N}r_{j}^{d}
<1<\sum_{j=1}^{N}r_{j}^{d-1}$.
Since the expressions in \eqref{eqn:SS} are pairwise disjoint, the Lebesgue measure of $G$ can be written as
\begin{equation}\label{Eq:G-meas}
m(G)=\sum_{j=1}^{N}r_{j}^{d}m(G)+m(G_{0}).
\end{equation}
The condition $\sum_{j=1}^{N}r_{j}^{d}<1<\sum_{j=1}^{N}r_{j}^{d-1}$ ensures that $G$ has a finite volume 
and there exists a unique number $\b\in (d-1,d)$ such that 
\[
\sum_{j=1}^{N}r_{j}^{\b}=1.
\]
It follows from \cite[Theorem A]{LV1996} that the number $\b$ is equal to the interior Minkowski dimension of $\partial G$. 
For more details, see \cite{LV1996} (pages 193--194).

We recall the following renewal theorem from \cite{LV1996}. 
For a given function $\phi:\R\to \R$, we look for the solution $f$ of the following \textit{renewal equation}
\beq\label{eqn:RE}
f=Lf +\phi,
\eeq
where $f:\R\to \R$, $L_{\gamma}f(z)=f(z-\gamma)$, $\gamma\in \R$, and $Lf(z)=\sum_{j=1}^{N}c_{j}L_{\gamma_{j}}f(z)=\sum_{j=1}^{N}c_{j}f(z-\gamma_{j})$ 
with $c_{j}>0$, $\gamma_{j}$ are distinct points in $\R$, and $\sum_{j=1}^{N}c_{j}=1$.
The following theorem is the renewal theorem on the solution of the renewal equation \eqref{eqn:RE}.
A set of finite real numbers $\{\gamma_{1},\cdots, \gamma_{N}\}$ is \textit{arithmetic}
 if $\frac{\gamma_{i}}{\gamma_{j}}\in \mathbb{Q}$ for all indices. The maximal number $\gamma$ 
 such that $\frac{\gamma_{i}}{\gamma}\in \mathbb{Z}$ is called the span of $\{\gamma_{1},\cdots, \gamma_{N}\}$.
If the set is not arithmetic, it is called \textit{non-arithmetic}.
\begin{thm}[Renewal Theorem \cite{LV1996}]\label{thm:RT}
Suppose that a map $f:\R\to \R$ satisfies the renewal equation \eqref{eqn:RE} with
\[
\lim_{z\to-\infty}f(z)=0,
\]
and 
\[
|\phi(z)|\leq c_{1}e^{-c_{2}|z|}, \quad z\in \R,
\]
for some constants $c_1,c_2>0$. Then, the solution of the renewal equation \eqref{eqn:RE} is given by
\[
f(z)=\sum_{n=0}^{\infty}L^{n}\phi(z)=\phi(z)+\sum_{n=1}^{\infty}\sum_{i_1,\cdots, i_n} 
c_{i_1}\cdots c_{i_n}L_{\gamma_{i_1}}\cdots L_{\gamma_{i_{n}}}\phi(z),
\]
where $i_1,\cdots, i_n$ take values from 1 to $N$ and are not necessarily different.
Furthermore, if $\{\gamma_{j}\}$ is non-arithmetic, then 
$$
f(z)=\frac{1}{\sum_{j=1}^{N}c_{j}\gamma_{j}}\int_{-\infty}^{\infty}\phi(x)dx +o(1) \text{ as } z\to\infty.
$$ 
If $\{\gamma_{j}\}$ is arithmetic with span $\gamma$, then
$$
f(z)=\frac{\gamma}{\sum_{j=1}^{N}c_{j}\gamma_{j}}\sum_{k=-\infty}^{\infty}\phi(z-k\gamma) +o(1) \text{ as } z\to\infty.
$$
\end{thm}

\section{Main Results} 
\label{Body}

In this section, we prove the main theorems of this paper, Theorems 
\ref{thm:higher alpha} and \ref{thm:lower alpha}, on the 
spectral heat content for isotropic stable processes on  a fractal drum $G$ 
that satisfies \eqref{eqn:SS} for some open set $G_0$. For each theorem, we 
need to impose the following conditions depending on whether $\alpha\in (d-\b, 2)$ or $\alpha\in (0,d-\b]$. 
\begin{cond}\label{cond:higher alpha}
\rm{
When $\alpha\in (d-\b,2)$, the following conditions hold. 

When $d=1$: 
\begin{enumerate}[(a)]
\item There exist disjoint open intervals $(a_0, b_0)$ and $(a_j,b_j)$, $j\in \{1,\cdots, N\}$, such that  $G_0\subset (a_0,b_0)$ and $R_{j}G\subset(a_j,b_j)$ for $1\leq j\leq N$.
\item $(a_0,b_0)\cap (G\setminus G_{0})=\emptyset$ and $(a_j,b_j)\cap (G\setminus R_{j}G)=\emptyset$.
\end{enumerate}

When $d\geq 2$, we assume:
\begin{enumerate}[(a)]
\item $G_0$ is a bounded $C^{1,1}$ open set.
\item For each $j\in \{1,2,\cdots, N\}$, there exist bounded $C^{1,1}$ open sets $\mathcal{O}_j$ such that $R_jG\subset \mathcal{O}_j$, $\mathcal{O}_i\cap \mathcal{O}_j=\emptyset$ for $i\neq j$,\footnote{This condition is related to, but not the same as the Open Set Condition (OSC) (cf. e.g., 
\cite[p.139]{Fal03}). More specifically, the OSC assumes that $\mathcal{O}_j = R_j (O)$ for some open set $O$ such that $\mathcal{O}_i\cap \mathcal{O}_j=\emptyset$ for $i\neq j$, but does not require $O$ or $R_j (O)$ to be $C^{1,1}$.} and $\mathcal{O}_j\cap G_0=\emptyset$ for all $j\in\{1,2,\cdots, N\}$.
\end{enumerate}
}
\end{cond}

\begin{cond}\label{cond:lower alpha}
When $\alpha\in (0,d-\b]$, we assume that the interior Minkowski dimension $\b$ of $\partial G$ and the Hausdorff 
dimension $\text{dim}_{\H}(\partial G)$ of $\partial G$ coincide and 
$
\H^{\b}(\partial G)>0.
$
\end{cond}

Here is a remark about Condition 3.2.  If we assume that 
\[
\partial (G\backslash G_0) = \bigcup_{j=1}^N R_j \partial (G \backslash G_0),
\]
i.e., $\partial (G\backslash G_0)$ is a self-similar set, and the open sets $R_j  G$ ($i = 1, \ldots, N$)
and $G_0$ are pairwise disjoint. Then the similitudes $R_1, \ldots, R_N$ satisfy the open set condition (cf. e.g., 
\cite[p.139]{Fal03}). 
Hence, we have  $\dim_{\H} \partial (G\backslash G_0)= \dim_{\cal M} \partial (G\backslash G_0)= \b$ and 
$\H^{\b}(\partial (G\backslash G_0))>0$, see \cite[Theorem 9.3]{Fal03}. 
Consequently, Condition 3.2 holds. By applying this, one can verify that both the complement of the Cantor set and 
the complement of the Sierpi\'nski gasket satisfy Condition 3.2.

We start with the following lemma on the scaling property of the spectral heat 
content $Q_{D}^{(\alpha)}(t)$, which is analogous to \cite[Lemma 2.7]{PX23}.
Since the proof follows from the scaling property of $X$ and is almost identical 
to that of \cite[Lemma 2.7]{PX23}, it will be omitted.
\begin{lemma}\label{lemma:scaling}
Let $R$ be a similitude with coefficient $r$ and $D$ is any open set in 
$\R^{d}$. Then, we have
$$
Q_{RD}^{(\alpha)}(t)=r^{d}Q_{D}^{(\alpha)}(t/r^{\alpha}), \quad t>0.
$$
\end{lemma}

The next lemma concerns the \textit{super-additivity} of the spectral heat content 
for stable processes. For the spectral heat content of subordinate killed Brownian 
motions, an exact additivity property holds under disjoint unions of domains (see 
\cite[Lemma 2.5]{PX23}). However, in the case of stable processes, the process can 
jump directly from one component of the domain to another without exiting it. This 
phenomenon causes the spectral heat content for stable processes to be super-additive rather than additive. The strict inequality in (\ref{super}) may hold and 
this is one of the main sources of difficulty in applying the renewal theorem.

\begin{lemma}\label{lemma:super additivity}
Let $D_{1}$ and $D_{2}$ be open sets in $\R^{d}$ with $D_{1}\cap D_{2}=\emptyset$. 
Then, for $\alpha\in(0,2)$
\begin{equation}\label{super}
Q^{(\alpha)}_{D_1\cup D_{2}}(t)\geq Q^{(\alpha)}_{D_1}(t)+Q^{(\alpha)}_{D_2}(t).
\end{equation}
\end{lemma}
\pf The proof is almost trivial, but we provide the details for the reader's convenience. Note that 
\begin{align*}
Q^{(\alpha)}_{D_1\cup D_{2}}(t)&=\int_{D_1\cup D_2}\P_{x}(\tau_{D_1\cup D_2}>t)dx
=\int_{D_1}\P_{x}(\tau_{D_1\cup D_2}>t)dx+\int_{D_2}\P_{x}(\tau_{D_1\cup D_2}>t)dx\\
&\geq \int_{D_1}\P_{x}(\tau_{D_1}>t)dx+\int_{D_2}\P_{x}(\tau_{D_2}>t)dx
=Q^{(\alpha)}_{D_1}(t)+Q^{(\alpha)}_{D_2}(t).
\end{align*}
\qed

To overcome the aforementioned difficulty, we estimate the probability that the
process jumps from one component of the domain to another and remains inside the domain up to time $t > 0$.
We first recall the following estimate on the measure of the set of points located at a fixed distance from the boundary of a $C^{1,1}$ open set.
Recall that an open set $D$ in $\R^d$ is said to be a $C^{1, 1}$ open set if 
there exist positive constants $R_0$ and $\Lambda_0$ such that, 
for every $z\in \partial D$, there exist a $C^{1, 1}$ function $\psi=\psi_z: \R^{d-1}\to \R$ satisfying $\psi(0)=0$, $\nabla \psi(0)
=(0, \cdots, 0)$, $\|\nabla\psi\|_\infty\le \Lambda_0$, $|\nabla \psi(x)-\nabla \psi(y)|\le \Lambda_0 |x-y|$ and an orthonormal coordinate 
system $CS_z: y=(\widetilde y, y_d)$
with origin at $z$ such that
\[
B(z, R_0)\cap D=B(z, R_0)\cap \{ y=(\widetilde y, y_d) \mbox{ in } CS_z: y_d>\psi(\widetilde y)\}.
\]
The pair $(R_0, \Lambda_0)$ is called a $C^{1,1}$ characteristic of the open set $D$, and $R_0$ is called a localization radius. 

\begin{lemma}{\cite[Lemma 5]{vB87}}\label{lemma:stability}
Let $D$ be a bounded $C^{1,1}$ open set in $\R^{d}$ with characteristic $(R_0, \Lambda_0)$ and define for $0\leq q <R_{0}$,
$$
D_{q}=\{x\in D : \delta_{D}(x)>q\}.
$$
Then
$$
\left(\frac{R_{0}-q}{R_{0}}\right)^{d-1}|\partial D|\leq |\partial D_{q}| \leq \left(\frac{R_{0}}{R_{0}-q}\right)^{d-1}|\partial D|,
\quad 0\leq q < R_{0}.
$$
\end{lemma}

Since $X$ is isotropic, the projection of $X$ onto any line passing through the 
origin has the same distribution and is a one-dimensional $\alpha$-stable 
process. Let $\vec{e}_1, \ldots, \vec{e}_d$ be the standard basis of $\R^d$ and 
let  $X^{(i)}_{t}=\langle X_t, \vec{e}_i\rangle$ ($i = 1, \ldots, d$) be the 
coordinate processes of $X$.
For $1 \le i \le d$, let $\overline{X^{(i)}}_{t} = \sup_{0 
\le s \le t}X^{(i)}_{s}$ be the supremum process of $X^{(i)}$. 
We recall the following estimate from \cite[Equation (3.3)]{P22}: There exists a 
constant $c=c(\alpha)>0$ such that for any $L>0$ and $t>0$
\beq\label{eqn:ub}
\P(\overline{X^{(i)}}_{t}>L)\leq \frac{ct}{L^{\alpha}}.
\eeq

Now we are ready to prove Lemma \ref{lemma:error high dim}.
\begin{lemma}\label{lemma:error high dim}
Let $\alpha\in (d-\b,2)$ and assume Condition \ref{cond:higher alpha} holds. 
Then, there exists $t_0>0$ such that for all $0<t\leq t_0$
\[
\max\left(\int_{R_jG}\P_{x}(\tau_{R_jG}< t \leq \tau_{G})dx, \,
\int_{G_0}\P_{x}(\tau_{G_0}\leq t < \tau_{G})dx\right) \leq cE(t),
\]
where $E(t)=t^{1/\alpha},\, t\ln(1/t)$, or $t$ for $\alpha\in (1,2)$, $\alpha=1$, or $\alpha\in (0,1)$, respectively. 
\end{lemma}
\pf
The proof for $d=1$ is almost the same as $d\geq 2$, so we will only provide the proof for $d\geq 2$.  
From Condition \ref{cond:higher alpha}, there exist $C^{1,1}$ open sets $\mathcal{O}_{j}$ such that $R_jG\subset \mathcal{O}_j$ for $1\leq j\leq N$.
Since it is enough to show this only for the case of $R_1G$, we use $RG$ and $\mathcal{O}$ instead of $R_1G$ and $\mathcal{O}_1$ for notational simplicity. First, note that $X_{\tau_{RG}}\in G\setminus RG$ on the event  $\{\tau_{RG}\leq t <\tau_{G} \}$ when the process starts on $x\in RG$. 
The condition  $\mathcal{O}\cap \mathcal{O}_j=\emptyset$ for all 
$\mathcal{O}_j$ such that $\mathcal{O}\neq \mathcal{O}_j$ and 
$\mathcal{O}\cap G_0=\emptyset$ implies that $\tau_{RG}=\tau_{\mathcal{O}}$ on the event  $\{\tau_{RG}\leq t <\tau_{G} \}$
and this implies that when the process starts on $RG$
\[
\{\tau_{RG}\leq t <\tau_{G} \} \subset \{\tau_{\mathcal{O}}\leq t\}.
\]

Let $t_0$ be small enough so that $\max(t_0,t_0^{1/\alpha})<R_0/2$, where 
$R_0$ is a localization radius, one of the $C^{1,1}$ characteristics of 
$\mathcal{O}$. Assume that $t\leq t_0$. 
We write $\O$ as
\[
\O=A(t)\cup B(t),
\]
where
\[
A(t)=\{x\in \O: \delta_{\O}(x)<t^{1/\alpha}\}, \quad B(t)=\O\setminus A(t).
\]

Hence, 
\begin{align}\label{eqn:aux1}
&\int_{RG}\P_{x}(\tau_{RG}< t \leq \tau_{G})dx \leq \int_{\O}\P_{x}(\tau_{RG}< t \leq \tau_{G})dx\nn\\
=&\int_{A(t)}\P_{x}(\tau_{RG}< t \leq \tau_{G})dx +\int_{B(t)}\P_{x}(\tau_{RG}< t \leq \tau_{G})dx.
\end{align}
For the first expression above, we use the trivial upper bound $\P_{x}(\tau_{RG}< t \leq \tau_{G})\leq 1$ and obtain
\begin{align*}
&\int_{A(t)}\P_{x}(\tau_{RG}< t \leq \tau_{G})dx\leq m(A(t)).
\end{align*}
It follows from Lemma \ref{lemma:stability}, the coarea formula, and the fact $\frac{R_0}{R_0-s}\leq \frac{R_0}{R_0-t^{1/\alpha}}\leq \frac{R_0}{R_0-R_0/2}=2$
\beq\label{eqn:aux2}
m(A(t))=\int_{0}^{t^{1/\alpha}}|\partial \O_s|ds\leq \int_{0}^{t^{1/\alpha}} \left(\frac{R_0}{R_0-s}\right)^{d-1}|\partial O|ds\leq 2^{d-1}|\partial \O|t^{1/\alpha}.
\eeq

Now we handle the second expression in \eqref{eqn:aux1}. We further divide $B(t)$ into
\[
B(t)=\{x\in \O: \delta_{\O}(x)\geq t^{1/\alpha}\} =\{x\in \O: t^{1/\alpha}\leq \delta_{\O}(x) <R_0/2\}  \cup \{x\in \O: \delta_{\O}(x)\geq R_0/2\}. 
\]
When the process starts on $\{x\in \O: \delta_{\O}(x)\geq R_0/2\}$, we have $\{\tau_{\O}\leq t\} \subset \bigcup_{i=1}^{d}\{\overline{X^{(i)}}_{t}\geq \frac{R_0}{2\sqrt{d}}\}$. 
Hence, it follows from \eqref{eqn:ub}
\begin{align}\label{eqn:aux3}
&\int_{\{x\in \O: \delta_{\O}(x)\geq R_0/2\}}\P_{x}(\tau_{RG}< t \leq \tau_{G})dx
\leq \int_{\{x\in \O: \delta_{\O}(x)\geq R_0/2\}}\P_{x}(\tau_{\O}< t)dx\nn\\
&\leq d\int_{\{x\in \O: \delta_{\O}(x)\geq R_0/2\}}\P\Big(\overline{X^{(1)}}_{t}\geq \frac{R_0}{2\sqrt{d}}\Big)
\leq cdm(\O)\frac{t}{(\frac{R_0}{2\sqrt{d}})^{\alpha}}.
\end{align}

Finally, for the set $\{x\in \O: t^{1/\alpha}\leq \delta_{\O}(x) <R_0/2\}$, we have $\{\tau_{\O}\leq t\} \subset \bigcup_{j=1}^{d}\{\overline{X^{(i)}}_{t}\geq \frac{\delta_{\mathcal{O}}(x)}{\sqrt{d}}\}$. This, together with Lemma \ref{lemma:stability}, \eqref{eqn:ub} and the coarea formula, implies that
\begin{align*}
&\int_{\{x\in \O: t^{1/\alpha}\leq \delta_{\O}(x) <R_0/2 \}}\P_{x}(\tau_{RG}< t \leq \tau_{G})dx
\leq\int_{\{x\in \O: t^{1/\alpha}\leq \delta_{\O}(x) <R_0/2 \}}\P_{x}(\tau_{\O}\leq t)dx\nn\\
&\leq d\int_{\{x\in \O: t^{1/\alpha}\leq \delta_{\O}(x) <R_0/2 \}}\P \Big(\overline{X^{(1)}}_{t}>\frac{\delta_{\O}(x)}{\sqrt{d}} \Big)dx
=d\int_{t^{1/\alpha}}^{R_0/2}|\partial \O_u|\P \Big(\overline{X^{(1)}}_{t}>\frac{u}{\sqrt{d}} \big)du\nn\\
&\leq cd2^{d-1}d^{\alpha/2}|\partial O|\int_{t^{1/\alpha}}^{R_0/2}\frac{t}{u^{\alpha}}du.
\end{align*} 

It is elementary to see that for $\alpha\in (1,2)$ and $t^{1/\alpha}<R_0/2$
\[
\int_{t^{1/\alpha}}^{R_0/2}\frac{t}{u^{\alpha}}du\leq \int_{t^{1/\alpha}}^{\infty}\frac{t}{u^{\alpha}}du \leq c_{1}t^{1/\alpha}.
\]
For $\alpha=1$ and $t< R_0/2$
\[
\int_{t}^{R_0/2}\frac{t}{u}du\leq  c_{2}t\ln(1/t),
\]
and for $\alpha\in (0,1)$ and $t^{1/\alpha}<R_0/2$
\[
\int_{t^{1/\alpha}}^{R_0/2}\frac{t}{u^{\alpha}}du\leq\int_{0}^{R_0/2}\frac{t}{u^{\alpha}}du\leq c_{3}t.
\]
This shows that there exists a constant $c>0$ such that
\beq\label{eqn:aux4}
\int_{\{x\in \O: t^{1/\alpha}\leq \delta_{\O}(x) <R_0/2 \}}\P_{x}(\tau_{RG}< t \leq 
\tau_{G})dx \leq cE(t).
\eeq
Hence, the conclusion follows from \eqref{eqn:aux2}, \eqref{eqn:aux3}, and \eqref{eqn:aux4}.
\qed

Now we are ready to state and prove Theorem \ref{thm:higher alpha}. 
For convenience, we introduce the following notations:
\beq\label{eqn:D and R}
\mathcal{D}(t)=Q_{G}^{(\alpha)}(t)-\sum_{j=1}^{N}Q_{R_{j}G}^{(\alpha)}(t)-Q_{G_{0}}^{(\alpha)}(t) \quad \text{ and } \quad 
\mathcal{R}(t)=\left(m(G_{0})-Q_{G_{0}}^{(\alpha)}(t)-\mathcal{D}(t)\right)t^{-\frac{d-\b}{\alpha}}.
\eeq

\begin{thm}\label{thm:higher alpha}
Let $d\geq 1$, $\alpha\in (d-\mathfrak{b}, 2)$, where $\b$ is the interior Minkowski
dimension of $\partial G$, and $G$ is a set given as in \eqref{eqn:SS}.
Assume that Condition \ref{cond:higher alpha} holds true. 
\begin{enumerate}
\item If $\{\ln\frac{1}{r_j}\}_{j=1}^{N}$ is non-arithmetic, then we have 
\begin{equation}\label{Eq:Th371}
Q_{G}^{(\alpha)}(t)=m(G)-C_{1}t^{\frac{d-\b}{\alpha}} +o(t^{\frac{d-\b}{\alpha}}) \text{ as } t\downarrow 0,
\end{equation}
where $C_{1}$ is the constant given by 
$$C_{1}=\frac{\int_{-\infty}^{\infty}\mathcal{R}(e^{-z})dz}{\sum_{j=1}^{N}r_{j}^{\b}\ln\frac{1}{r_{j}^{\alpha}}}=
\frac{\int_{0}^{\infty}\left(m(G_{0})-Q_{G_{0}}^{(\alpha)}(u)-\mathcal{D}(u)\right)u^{-1-\frac{d-\b}{\alpha}}dt}{\sum_{j=1}^{N}r_{j}^{\b}\ln\frac{1}{r_{j}^{\alpha}}}.
$$

\item If $\{\ln\frac{1}{r_{j}}\}_{j=1}^{N}$ is arithmetic with span $\rho$, then we have 
\begin{equation}\label{Eq:Th372}
Q_{G}^{(\alpha)}(t)=m(G)-f(-\ln t)t^{\frac{d-\b}{\alpha}} +o(t^{\frac{d-\b}{\alpha}}) \text{ as } t\downarrow 0,
\end{equation}
where 
\[
f(z)=\frac{\alpha\rho}{\sum_{j=1}^{N}r_{j}^{\b}\ln\frac{1}{r_{j}^{\alpha}}}
\sum_{n=-\infty}^{\infty}\mathcal{R}(e^{-(z-n\alpha\rho)}).
\]
\end{enumerate}
\end{thm}

\pf
The proof for the case of $d=1$ is almost the same as for $d\geq 2$, so we omit the 
proof for $d=1$. 
The proof for $d \ge 2$ is similar in spirit to the proof of \cite[Theorem 3.3]
{PX23}, but the spectral heat content $Q_{G}^{(\alpha)}(t)$ for stable 
processes is super-additive instead of additive under disjoint union, one 
must make the necessary modifications. We write $Q_{G}^{(\alpha)}(t)$ as
\begin{align*}
&Q_{G}^{(\alpha)}(t)=\sum_{j=1}^{N}Q_{R_{j}G}^{(\alpha)}(t)+Q_{G_{0}}^{(\alpha)}
(t)+\left(Q_{G}^{(\alpha)}(t)-\sum_{j=1}^{N}
Q_{R_{j}G}^{(\alpha)}(t)-Q_{G_{0}}^{(\alpha)}(t)\right)\\
&= \sum_{j=1}^{N}Q_{R_{j}G}^{(\alpha)}(t)+Q_{G_{0}}^{(\alpha)}(t)+\mathcal{D}(t).
\end{align*}

It follows from Lemma \ref{lemma:scaling} that 
$$
Q_{G}^{(\alpha)}(t)=\sum_{j=1}^{N}r_{j}^{d}Q_{G}^{(\alpha)}(\frac{t}{r_{j}^{\alpha}}) + m(G_0)- (m(G_0)-Q_{G_{0}}^{(\alpha)}(t))+\mathcal{D}(t).
$$
We write $Q_{G}^{(\alpha)}(t)$ as $Q_{G}^{(\alpha)}(t)=m(G)-f(-\ln t)t^{\frac{d-\b}
{\alpha}}$ for some function $f(\cdot)$. Then, we can rewrite the equation above as 
\begin{align*}
&m(G)-f(-\ln t)t^{\frac{d-\b}{\alpha}}\\
=&\sum_{j=1}^{N}r_{j}^{d}\bigg(m(G)-f(-\ln\frac{t}{r_{j}^{\alpha}})\bigg(\frac{t}{r_{j}^{\alpha}}\bigg)^{\frac{d-\b}{\alpha}}\bigg) +m(G_0)- (m(G_0)-Q_{G_{0}}^{(\alpha)}(t))+\mathcal{D}(t).
\end{align*}
By the change of variable $z=-\ln t$ and (\ref{Eq:G-meas}), we arrive at the following renewal equation
\[
f(z)=\sum_{j=1}^{N}r_{j}^{\b}f(z-\ln\frac{1}{r_{j}^{\alpha}}) +\left(m(G_0)-
Q_{G_0}^{(\alpha)}(e^{-z})-\mathcal{D}(e^{-z})\right)e^{\frac{(d-\b)z}{\alpha}}.
\]
Let
\[
\mathcal{R}(e^{-z}) = \left(m(G_0)-Q_{G_0}^{(\alpha)}(e^{-z})-\mathcal{D}(e^{-z})\right)e^{\frac{(d-\b)z}{\alpha}}.
\]
We proceed to show that there exist constants $c_1$ and $c_2$ such that
\begin{equation}\label{Eq:R}
\mathcal{R}(e^{-z})\le c_1 e^{-c_2 |z|}, \ \ \hbox{ for } \ z \in \R.
\end{equation}

By repeating the same argument as \cite[Lemma 3.1]{PX23} using \cite[Theorem 1.1]{PS22} instead of \cite[Theorem 1.1]{PS19} we conclude that 
$$
0\leq \left(m(G_0)-Q_{G_0}^{(\alpha)}(e^{-z})\right)e^{\frac{(d-\b)z}{\alpha}} \leq ce^{-c|z|} \ \ \text{ for all }z\in \R,
$$
for some constant $c=c(\alpha,\b,d,m(G_0))>0$.
It remains to prove that $\mathcal{D}(e^{-z})e^{\frac{(d-\b)z}{\alpha}}$ decays exponentially as $z\to\infty$, so that one can apply the Renewal Theorem, Theorem \ref{thm:RT}.
By Lemma \ref{lemma:super additivity}, we clearly have $\mathcal{D}(t)\geq 0$. Note 
that 
\begin{align*}
&\mathcal{D}(t)=Q_{G}^{(\alpha)}(t)-\sum_{j=1}^{N}Q_{R_{j}G}^{(\alpha)}(t)-
Q_{G_{0}}^{(\alpha)}(t)\\
=&\int_{\bigcup_{j=1}^{N}R_{j}G\cup G_{0}}\P_{x}(\tau_{G}^{(\alpha)}>t)dx-\sum_{j=1}^{N}\int_{R_{j}G}\P_{x}(\tau_{R_{j}G}>t)dx-\int_{G_{0}}\P_{x}(\tau_{G_0}>t)dx\\
=&\sum_{j=1}^{N}\int_{R_{j}G}\left(\P_{x}(\tau_{G}>t)-\P_{x}(\tau_{R_{j}G}>t)\right)dx+\int_{G_0}\left(\P_{x}(\tau_{G}>t)-\P_{x}(\tau_{G_0}>t)\right)dx\\
=&\sum_{j=1}^{N}\int_{R_{j}G}\P_{x}(\tau_{R_{j}G}\leq t< \tau_{G})dx+\int_{G_0}\P_{x}(\tau_{G_0}\leq t<\tau_{G})dx.
\end{align*}
It follows from Lemma \ref{lemma:error high dim} that $\mathcal{D}(t)\leq cE(t)$. This proves (\ref{Eq:R}). Hence, the conclusions (\ref{Eq:Th371}) and  (\ref{Eq:Th372}) follow from the Renewal Theorem 2.1 immediately. 
\qed

Next, we investigate $Q_{D}^{(\alpha)}(t)$ when $\alpha\in (0,d-\b]$. In this case, 
we impose Condition \ref{cond:lower alpha}. 

\begin{thm}\label{thm:lower alpha}
Let $G$ be an open set given as in \eqref{eqn:SS}, $\alpha\in(0,d-\b]$, and Condition \ref{cond:lower alpha} is satisfied. Then, 
\begin{equation}\label{Eq:Per}
\lim_{t\to 0}\frac{m(G)-Q_{G}^{(\alpha)}(t)}{t}=\textup{Per}^{(\alpha)}(G)=\int_{G}\int_{G^c}\frac{c(\alpha, d)}{|x-y|^{d+\alpha}}dydx,
\end{equation}
where the constant $c(d,\alpha) $ is given by 
\begin{equation}\label{Eq:Cd}
c(\alpha, d) = \frac{\alpha \Gamma\big(\frac{d +\alpha} 2\big)} {2^{1 - \alpha} \pi^{d/2}\Gamma\big(1 - \frac \alpha 2\big)}.
\end{equation}
\end{thm}
\pf
By Frostman's theorem (cf. \cite[Page 133]{Ka85} or \cite{Fal03}) and 
Remark 3 in \cite[Page 134]{Ka85}, we derive that for every $\alpha\in (0,d-\b]$, 
we have  ${\rm Cap}_{d-\alpha}( \partial G) = 0$, where ${\rm Cap}_{d-\alpha}$ 
denotes the $(d-\alpha)$-dimensional Riesz-Bessel capacity. Hence, it 
follows from \cite[Lemma 7]{T66} that $\partial G$ is polar for $X$. 
Now the conclusion follows immediately from \cite[Theorem 3.4]{GPS19}.
\qed

\begin{remark}\label{remark:mistake}
\rm{
We remark that there is a minor error in \cite[Corollary 3.7]{GPS19}. There, the authors claimed that $\lim_{t\to0}\frac{|D|-Q_{D}^{X}(t)}{t}
=\textup{Per}^{X}(D)$ provided the underlying L\'evy process $X$ has bounded variation and $D$ is any open set with finite Lebesgue measure. 
An $\alpha$-stable processes has bounded variation if and only if $\alpha\in(0,1)$ and $\textup{Per}^{(\alpha)}(0,1)
=\textup{Per}^{(\alpha)}(\mathfrak{D})$, where $\mathfrak{D}$ is the complement of the ternary Cantor set since $m((0,1)\setminus \mathfrak{D})=0$. This violates Theorem \ref{thm:higher alpha} when
$\alpha\in (d-\b,1)$. The authors in \cite{GPS19} claimed the boundary of any open set in $\R$ is polar for the process, as a single point is polar. However, the number of boundary points of an open set in $\R$ can be uncountable, such as $\mathfrak{D}$, and it can be non-polar for $X$. 
}
\end{remark}

For comparison purpose, we now consider the \textit{regular heat content} (RHC) for $X$ on a Borel set $D \subset \R^d$, which is defined as 
\[
H_{D}^{(\alpha)}(t)=\int_{D}\P_{x}(X_{t} \in D)\,dx.
\]

For any Borel sets $D \subset \Omega \subset \R^d$ such that $m( 
\Omega\backslash D) = 0$, where $m$ is the Lebesgue measure in $\R^d$,
the following lemma shows that the regular heat content on $D$ is equal to 
that of $\Omega$.
\begin{lemma}\label{lemma:same1}
Let $D \subset \Omega$ be Borel sets in $\R^d$ such that 
$m(\Omega\backslash D) = 0$ and $\alpha\in (0,2)$. Then for all $t > 0$,
$$
H_{D}^{(\alpha)}(t)=H_{\Omega}^{(\alpha)}(t).
$$
\end{lemma}
\pf
Note that $\E[e^{i\xi X_{t}}]=e^{-t|\xi|^{\alpha}}\in L^{1}(\R)$ and this 
shows that $X$ has a continuous and bounded transition density $p^{(\alpha)}
(t,x)$ (heat kernel). We write $p^{(\alpha)}(t,x,y):=p^{(\alpha)}(t,y-x)$. 
Then
\begin{equation}\label{Eq:RHC1}
\begin{split}
H_{\Omega}^{(\alpha)}(t)&=\int_{\Omega}\P_{x}(X_{t}\in \Omega)dx
= \int_{D}\P_{x}(X_{t}\in \Omega)dx 
\end{split}
\end{equation}
because $m(\Omega\backslash D) = 0$. On the other hand, this last assumption 
also implies
\begin{equation}\label{Eq:RHC2}
\P_{x}(X_{t}\in \Omega )=\P_{x}(X_{t}\in D)+\P_{x}(X_{t}\in \Omega \backslash D)
=\P_{x}(X_{t}\in D).
\end{equation}
Combining (\ref{Eq:RHC1}) and  (\ref{Eq:RHC2}) yields
\begin{align*}
&H_{\Omega}^{(\alpha)}(t)=\int_{D}\P_{x}(X_{t}\in D)\,dx
= H_{D}^{(\alpha)}(t).
\end{align*}
This proves the lemma.
\qed

The small-time behavior of RHC $H_{\Omega}^{(\alpha)}(t)$ for a $C^{1,1}$ open set $\Omega$ was studied by Acu\~{n}a Valverde \cite{Val2017}. 
In particular, Theorems 1.1 and 1.2 in \cite{Val2017} showed that, for $d = 1$ and $\Omega = (a, b)$, and for $d\ge 2$ and $\Omega$ being a bounded $C^{1,1}$ open set, the following statements hold: 
\begin{itemize}   
\item[(i)] For $1 < \alpha< 2$,  
\begin{equation}\label{Eq:Th12a}
\lim_{t \to 0}\frac{m(\Omega) - H_{\Omega}^{(\alpha)}(t)}{t^{1/\alpha}} = 
\frac 1 \pi \Gamma\Big(1 - \frac 1 \alpha\Big) {\cal H}^{d-1}(\partial \Omega)
\end{equation}
\item [(ii)] For $\alpha = 1$,
\begin{equation}\label{Eq:Th12b}
\lim_{t \to 0}\frac{m(\Omega) - H_{\Omega}^{(1)}(t)}{t \log (1/t)} = 
\frac 1 \pi   {\cal H}^{d-1}(\partial \Omega).
\end{equation}
\item [(iii)] For $0 < \alpha< 1$, 
\begin{equation}\label{Eq:Th12c}
\lim_{t \to 0}\frac{m(\Omega) - H_{\Omega}^{(\alpha)}(t)}{t} 
= \textup{Per}^{(\alpha)}(\Omega).
\end{equation}
\end{itemize}
In the above, ${\cal H}^{0}(\partial \Omega) = \#(\partial \Omega)$ which is the cardinality of $\partial \Omega$, and  $\textup{Per}^{(\alpha)}(\Omega)$ is defined as in \eqref{Eq:Per}.
We remark that \eqref{Eq:Th12a} and \eqref{Eq:Th12b} follow from \cite[Theorems 1.1]{Val2017} directly and \eqref{Eq:Th12c} follows from \cite[Theorem 3.2]{GPS19}.

By combining \eqref{Eq:Th12a}-\eqref{Eq:Th12c} with Lemma \ref{lemma:same1}, we obtain the following corollary of Theorems 1.1 and 1.2 in \cite{Val2017} that is applicable to fractal drums. 
\begin{corollary}\label{Cor:RHC}
Let $\Omega = (a, b)$ if $d = 1$ and $\Omega \subset \R^d$ be a bounded $C^{1,1}$ open set if $d\ge 2$. Then for any fractal drum $G\subset \Omega$ that satisfies $m(\Omega\backslash G) = 0$, \eqref{Eq:Th12a}, \eqref{Eq:Th12b}, and \eqref{Eq:Th12c} hold with $\Omega$ replaced by $G$.
\end{corollary}

Theorem \ref{thm:higher alpha} and Corollary \ref{Cor:RHC} show that, when $\alpha \in (d-\mathfrak{b},2)$, the spectral heat content and the regular heat content on fractal drums decay at different rates. In this regime, the spectral heat content decays faster than the regular heat content because of the additional cooling effect along the fractal boundary of $D$. 
This behavior is very different from the corresponding results for regular and spectral heat contents on domains with smooth boundary \cite{Val2016, Val2017, LP25, PS22}, where the regular and spectral heat contents have the same decay rate but different leading coefficients.
In contrast, when $\alpha \in (0,d-\mathfrak{b}]$, Theorem \ref{thm:lower alpha} and 
\eqref{Eq:Th12c} show that the spectral heat content and the regular heat content have not only the same decay rate, but also the same leading coefficients. This is because, in this range  $\alpha \in (0,d-\mathfrak{b}]$, the fractal boundary of $D$ is invisible to the stable process.

\section{Case Studies: the Complement of the Ternary Cantor Set and the modified Sierpi\'nski Gasket}\label{example}

In this section, we provide two examples where Theorems \ref{thm:higher alpha} and \ref{thm:lower alpha} can be applied to find the small-time asymptotic behavior of the spectral heat content. 
Let $\mathfrak{D}$ be the complement of the ternary Cantor set in $[0,1]$ and $\mathfrak{C}$ be the ternary Cantor set so that
\[
\mathfrak{C}\cup \mathfrak{D}=[0,1] \text{ and }  \mathfrak{C} \cap\mathfrak{D}=\emptyset.
\]
Note that $\mathfrak{D}$ can be decomposed as 
\[
\mathfrak{D}= R_1\mathfrak{D}\cup R_2\mathfrak{D} \cup \mathfrak{D}_0,
\]
where $R_1$ and $R_2$ are two similitudes on $\R$ defined by $R_1(x) = \frac{x}{3}$  and  $R_2(x) = \frac{x}{3} + \frac23$ and $\mathfrak{D_0} = (\frac13, \frac23)$. 
It is well-known that the interior Minkowski dimension and the Hausdorff dimension of $\mathfrak{C}$ coincide and are equal to $\frac{\log 2}{\log 3}$.
We find the small-time asymptotic behavior of the regular heat content as well as the spectral heat content on $\mathfrak{D}$ and show that their behaviors are quite different. 

In the following, the case for $\alpha \in (1-\frac{\ln 2}{\ln 3},2)$ follows from Theorem \ref{thm:higher alpha} with $r_{1}=r_{2}=\frac13$, 
$\b=\frac{\ln 2}{\ln 3}$, and $\rho=1$; and the case for $\alpha \in (0,1-\frac{\ln 2}{\ln 3}]$ follows from Theorem \ref{thm:lower alpha}.

\begin{example}\label{ex:TC}
Let $\mathfrak{D}$ be the complement of the ternary Cantor set in $[0,1]$. Then,
\begin{enumerate}
\item For $\alpha \in (1-\frac{\ln 2}{\ln 3},2)$, 
$$
Q_{\mathfrak{D}}^{(\alpha)}(t)=m(\mathfrak{D})-F(-\ln t)t^{\frac{1-\frac{\ln2}{\ln 3}}
{\alpha}} +o(t^{\frac{1-\frac{\ln 2}{\ln 3}}{\alpha}}) \text{ as } t\downarrow0,
$$
where $F(z)=\frac{1}{\ln 3}\sum_{n=-\infty}^{\infty}\tilde{R}(z-n\alpha)$ with $\tilde{R}(t)=\left(|\mathfrak{D_0}|-Q_{\mathfrak{D_0}}^{(\alpha)}(t)-\tilde{A}(t)\right)t^{-\frac{1-\frac{\ln 2}
{\ln 3}}{\alpha}}$ and $\tilde{A}(t)=Q_{\mathfrak{D}}^{(\alpha)}(t)-\sum_{j=1}^{2}Q_{R_{j}\mathfrak{D}}^{(\alpha)}(t)-Q_{\mathfrak{D_0}}^{(\alpha)}(t)$.

\item For $\alpha\in (0,1-\frac{\ln 2}{\ln 3}]$,
$$
Q_{\mathfrak{D}}^{(\alpha)}(t)=m(\mathfrak{D})-\textup{Per}^{(\alpha)}(\mathfrak{D})t +o(t), 
$$
where $\textup{Per}^{(\alpha)}(\mathfrak{D})=\int_{\mathfrak{D}}\int_{\mathfrak{D}^{c}}\frac{c(1,\alpha)}{|x-y|^{1+\alpha}}dydx$.
\end{enumerate}
\end{example}

We turn our attention to the regular heat content for $X$ on $\mathfrak{D}$.
Recall from Lemma \ref{lemma:error high dim} that $E(t)=t^{1/\alpha}, t\ln(1/t)$, or $t$ for $\alpha\in (1,2)$, $\alpha=1$, or $\alpha\in (0,1)$, respectively. 
From  \eqref{Eq:Th12a}, \eqref{Eq:Th12b}, \eqref{Eq:Th12c}, and Corollary \ref{Cor:RHC}, we obtain the following result.

\begin{example}\label{ex:same}
Let $\mathfrak{D}$ be the complement of the ternary Cantor set in $[0,1]$. Then, 
$$
\lim_{t\to 0}\frac{m(\mathfrak{D})-H_{\mathfrak{D}}^{(\alpha)}(t)}{E(t)}=\lim_{t\to 0}\frac{m((0,1))-H_{(0,1)}^{(\alpha)}(t)}{E(t)}=
\begin{cases}
\frac{2}{\pi}\Gamma(1-\frac{1}{\alpha})&\text{if } \alpha\in (1,2),\\
\frac{2}{\pi}&\text{if } \alpha=1,\\
\textup{Per}^{(\alpha)}(0,1), &\text{if } \alpha\in (0,1),
\end{cases}
$$
where $\textup{Per}^{(\alpha)}(0,1):=\int_{(0,1)}\int_{(0,1)^{c}}\frac{A_{\alpha,1}}{|x-y|^{1+\alpha}}dydx$,
and $A_{\alpha,d}
=\alpha2^{\alpha-1}\pi^{-1-d/2}\sin(\frac{\pi\alpha}{2})\Gamma(\frac{d+\alpha}{2})\Gamma(\frac{\alpha}{2})$.
\end{example}

\begin{remark}
We remark that Theorem \ref{thm:higher alpha} and Example \ref{ex:same} show 
that the spectral heat content for fractal drums is very different 
from the SHC for a smooth domain such as $(0,1)$. 
In \cite[Theorem 1.1]{Val2016} the author 
showed that for $\alpha \in (0,1]$ 
$$
\lim_{t\to 0}\frac{m((0,1))-H^{(\alpha)}_{(0,1)}(t)}{E(t)}=\lim_{t\to 0}\frac{m((0,1))-Q^{(\alpha)}_{(0,1)}(t)}{E(t)}.
$$
However, as one can see in Theorem \ref{thm:higher alpha} the order of decay for $Q_{\mathfrak{D}}^{(\alpha)}(t)$ is $t^{\frac{d-\b}{\alpha}}$ 
when $\alpha \in (d-\b,1]$. Hence, for $\alpha\in (d-\b,1]$ 
we conclude that $H_{\mathfrak{D}}^{(\alpha)}(t)$ and $Q_{\mathfrak{D}}^{(\alpha)}(t)$ have different decay rates and one cannot infer the small-time asymptotic behavior
 of $Q_{\mathfrak{D}}^{(\alpha)}(t)$ from $H_{\mathfrak{D}}^{(\alpha)}(t)$.
\end{remark}

\begin{remark}
\rm{
More general Cantor-type sets were introduced in \cite[Chapter 4]{M95}. 
Let $E_{i_1,i_2,\cdots, i_{k}}$, $i_{j}\in \{1,2,\cdots, N\}$ be compact sets in $\R^{d}$ with the following properties:
\begin{enumerate}
\item $E_{i_1,i_2,\cdots, i_{k}, i_{k+1}}\subset E_{i_1,i_2,\cdots, i_{k}}$ for all $k\in \N$. 
\item $\max_{i_1,\cdots, i_{k}}d(E_{i_1,i_2,\cdots, i_{k}})\to 0$ as $k\to \infty$, where $d(E)$ is the diameter of $E$. 
\item $\sum_{j=1}^{N}d(E_{i_1,\cdots, i_{k}, j})^{s}=d(E_{i_1,\cdots, i_k})$.
\item For any ball $B$, $\sum_{B\cap E_{i_1,i_2,\cdots, i_{k}} \neq\emptyset }d(E_{i_1,i_2,\cdots, i_{k}})^{s}\leq c d(B)^s$ for some constant $c>0$. 
\end{enumerate}
Then, it is shown in \cite[Chapter 4]{M95}
\[
0<\H^{s}\bigg(\bigcap_{k=1}^{\infty}\bigcup_{i_1,\cdots, i_k}E_{i_1,i_2,\cdots, i_{k}}\bigg)<\infty.
\]
}
\end{remark}

\begin{remark}
\rm{
For more general Cantor-type sets, Ohtsuka \cite[Page 151]{Oh57} provided necessary and sufficient conditions for their $\beta$-dimensional Riesz-Bessel capacity to be 0.
Hence, we can extend Example \ref{ex:TC} to these Cantor-type sets as well. 
}
\end{remark}

As we mentioned earlier, the complement of the Sierpi\'nski gasket satisfies (\ref{eqn:SS}). For the reader's convenience, we recall the definition of the Sierpi\'nski gasket and its properties.
Let $T_{0}$ be the closed equilateral triangle with vertices $(0,0), (1, 0)$, and $(\frac12, \frac{\sqrt{3}}{2})$. Let $R_{1}(x)=\frac{x}{2}$, $R_{2}(x)=\frac{x}{2}+(\frac12, 0)$, and $R_{3}(x)=\frac{x}{2}+(\frac14,\frac{\sqrt{3}}{4})$. For $i\in \{1,2,3\}$, we define $T_{i}=R_{i}(T_0)$; and $T_{i_1,\cdots, i_{k}, i}=R_{i}(T_{i_1,i_2,\cdots, i_{k}})$ for $(i_1,\cdots, i_{k}) \in \{1, 2, 3\}^k$. The Sierpi\'nski gasket $\mathfrak{G}$ is defined by
\[
\mathfrak{G}=\bigcap_{k=1}^{\infty}\bigcup_{i_1,\cdots, i_{k}}T_{i_1,\cdots, i_k}.
\]
It is well known (cf. \cite{Fal03}) that the Hausdorff dimension of $\mathfrak{G}$ is $\frac{\ln 3}{\ln 2}$ and 
\[
0<\H^{\frac{\ln 3}{\ln 2}}(\mathfrak{G})<\infty.
\]
Define $G=T_{0}\setminus \mathfrak{G}$. Then, $G$ can be written as 
\[
G=R_{1}G\cup R_{2}G\cup R_{3}G\cup V,
\]
where $V$ is the open triangle whose vertices are $(1/4, \sqrt{3}/4)$, $(1/2,0)$ and $(3/4,\sqrt{3}/4)$. Then, $G$ is a set of the form in \eqref{eqn:SS} with $\sum_{j=1}^{3}(\frac12)^{\frac{\ln 3}{\ln 2}}=1$ and $\partial G=\mathfrak{C}$.
It follows from \cite[Theorem A]{LV1996} that the interior Minkowski dimension of $\mathfrak{C}$ is also $\frac{\ln 3}{\ln 2}$.

Hence, $G$ satisfies Condition \ref{cond:higher alpha} except that the triangle $V$ is not $C^{1,1}$. The asymptotic properties of the 
spectral heat content for $X$ on triangles, or more generally, 
domains with polygonal boundaries, have not been established at this point. Therefore, Theorem \ref{thm:higher alpha} cannot be directly applied to $G$.

In the following, we construct an open set $G'$, which is called the complement of a smoothed Sierpi\'nski gasket, such that 
Theorem \ref{thm:higher alpha} can be applied. This is done by modifying the construction of the Sierpi\'nski gasket by smoothing out the corners of the initially 
removed triangle [i.e., $V$], so that it becomes a $C^{1,1}$ open set. 
We illustrate the precise construction of the modified Sierpi\'nski gasket $\mathfrak{G'}$.
Define
\[
T_1'=\bigcup_{j=1}^{3}R_{j}T_0\cup C,
\]
where $C$ is a union of three compact regions near the three corners of the triangle $V$, so that $T_0\setminus T_1'$ is a $C^{1,1}$ open set.
Define $T'_{1,i_2,\cdots, i_{k}, i}=R_{i}(T'_{1,i_2,\cdots, i_{k}})$ for $i\in \{1,2,3\}$ and $k\geq 2$.
Now we define the modified Sierpi\'nski gasket $\mathfrak{G'}$ as
\[
\mathfrak{G'}=\bigcap_{k=2}^{\infty}\bigcup_{i_2,\cdots, i_{k}}T'_{1,i_2,\cdots, i_k}.
\]
Finally, we define $G'=T_0\setminus \mathfrak{G'}$.  Note that $\mathfrak{G'}$ and $G'$ satisfy
\beq\label{eqn:inhomo}
\mathfrak{G'}=\bigcup_{j=1}^{3}R_{j}\mathfrak{G'}\cup C,
\eeq
and
\[
G'=\bigcup_{j=1}^{3}R_{j}G'\cup (T_0\setminus T'_1).
\]
Note that sets of the form \eqref{eqn:inhomo} are called the inhomogeneous self-similar set with condensation $C$ in \cite{Barnsley06}.
The set $G'$ can be written as in \eqref{eqn:SS} with $G_0 = T_0\setminus T'_1$ being $C^{1,1}$. Hence, Theorem \ref{thm:higher alpha} holds for $G'$.

Although we believe that Theorem 3.7 should also hold for the complement of the Sierpi\'nski gasket $\mathfrak{G}$ and many other fractal drums such 
as the von Koch snowflake, we currently lack a method to prove it. We therefore state the following open question for future research.
\\

\noindent\textbf{Open Question}: Does the small-time asymptotic behavior as in Theorem \ref{thm:higher alpha} hold for the complement of  
Sierpi\'nski gasket $\mathfrak{G}$ and the von Koch snowflake?

\vspace{0.3in}

\noindent\textbf{\LARGE{Declarations}}
\vspace{0.3in}

\noindent\textbf{\large{Ethical Approval}}
Not applicable
\vspace{0.1in}

\noindent\textbf{\large{Funding}}
Hyunchul Park's research was supported in part by the AMS-Simons Research Enhancement Grants for Primarily Undergraduate Institution (PUI) 
Faculty 2025-2028 and the Research and Creative Projects Awards 2025-2026 from SUNY New Paltz.
Yimin Xiao's research was supported in part by the NSF grant DMS-2153846.
\vspace{0.1in}

\noindent\textbf{\large{Availability of data and materials }}
No datasets were generated or analyzed during the current study.
\vspace{0.1in}

\begin{singlespace}

\end{singlespace}
\end{doublespace}

\vskip 0.3truein

\noindent {\bf Hyunchul Park}

\noindent Department of Mathematics, State University of New York at New Paltz, NY 12561, USA

\noindent E-mail: \texttt{parkh@newpaltz.edu}

\bigskip

\noindent{\bf Yimin Xiao}

\noindent Department of Statistics and Probability, Michigan State University, East Lansing, MI 48824, USA

\noindent E-mail: \texttt{xiaoyimi@stt.msu.edu}

\end{document}